\documentclass{amsart}
\usepackage{amssymb,amsmath}
\usepackage{amsfonts}
\usepackage{pstricks, pst-node}
\usepackage[all]{xy}
\usepackage{array}
\numberwithin{equation}{section}
\theoremstyle{plain}
\newtheorem{lemma}{Lemma}
\newtheorem{corollary}{Corollary}
\begin{document}
\title{About Some Family of Elliptic Curves}
\author{K.~M.Bugajska}
\address{Department of Mathematics and Statistics,
York University,
Toronto, ON, M3J 1P3}
\email{bugajska@yorku.ca}
\date{\today}
\begin{abstract}
We examine the moduli space $\mathcal{E}\cong{\textbf{T}^{*}}$ of complex tori ${\textbf{T}(\tau)}\cong{\mathbb{C}/L(\tau)}$ where ${L_(\tau)}={const\cdot{\eta^2(\tau)L_{\tau}}}$. We find that the Dedekind  eta function furnishes a bridge between the euclidean and hyperbolic structures on $\textbf{T}^{*}\cong{\mathbb{C}-L_0/L_0}$ as well as between the doubly periodic Weierstrass function $\wp$ on $\textbf{T}^{*}$ and the theta function for the lattice $E_8$. The former one allows us to rewrite the Lame equation for the Bers embedding of $\mathcal{T}_{1,1}$ in a new form. We show that $L_0$ has natural decomposition into 8 sublattices (each equivalent to $L_0$), together with appropriate half-points and that this leads to local functions of the form $\vartheta^8_l(0,\tau_{\alpha})$ for a local map $(U_{\alpha},\tau_{\alpha})$ and to a relation with $E_8$.
\end{abstract}

\subjclass[2010]{Primary 11F06; Secondary 30C99}

\maketitle

\section{Introduction}
 We have shown in ~\cite{KMA10}  that the natural algebraic structures associated to the punctured torus $\textbf{T}^{*}\cong{\textit{H}/\Gamma'}$, (here $\Gamma'$ is the commutator group of the modular group ${\Gamma}=SL_{2}\mathbb{Z}$,   ${\Gamma'}=[\Gamma,\Gamma]$) viewed as the Veech modular curve of complex tori, produce exactly the generating matrix for the binary error correcting Golay code $\textit{G}_{24}$. This is a reason why in this paper we investigate the (on the other hand well known) punctured torus $\textbf{T}^{*}$ more carefully. We will find that the Dedekind eta function $\eta$ plays very important role. It furnishes not only a bridge between the hyperbolic and euclidean geometries on $\textbf{T}^{*}$ but it also  connects (see the formula $(5.5)$) the doubly-periodic Weierstrass function $\wp(p(z_{\alpha}),L_0)$ on $\textbf{T}^{*}$ with the theta function for the lattice $E_8$, that is with ${\Theta_{E_8}(\tau_{\alpha})}={\sum_{m=0}^{\infty}r_{E_8}(m)q_{\alpha}^m}$ (here  $\tau_{\alpha}={\tau_{\alpha}(x')}$, $z_{\alpha}=p(\tau_{\alpha})$,  ${x'}\in{U_{\alpha}}\subset{\textbf{T}^{*}}$ and $r_{E_8}(m)$ is the number of elements $\underline{v}\in{E_8}$ such that $\underline{v}\cdot{\underline{v}}=2m$).  

Since the Veech modular curve $\textbf{T}^{*}$ naturally carries the modular $\textit{J}$-invariant we may view each of the objects $\textit{G}_{24}$ and $E_8$ as a sort of a hidden structure associated to the Klein $\textit{J}$-function that is encoded in the projection $\textit{J}: {\textbf{T}^{*}}\rightarrow{Y(1)}={\textit{H}/\Gamma}$.  

In ~\cite{KMB10} we have shown that, similarly to strong consequences coming from  relations between $\Gamma'$ and the subgroups $\Gamma(2)$,  $\Gamma(3)$,  $\Gamma_c$ and  $\Gamma^{+}_{ns}(3)$ of the modular group $\Gamma$ (and investigated in this note) the relations between  ${\Gamma}=SL_{2}\mathbb{Z}$ and $\Gamma_0(p)$ (for the supersingular primes) introduce  a hidden structure asociated to the $\textit{J}$-function whose the full symmetry group $\mathcal{K}$ must have the order that is devided by each of these primes $p$. Since the full automorphism group of $\textit{G}_{24}$  (given by the Matieu group $\mathcal{M}_{24}$) must be a subgroup of $\mathcal{K}$, the conditions that $p||\mathcal{K}|$ together with the requirement that $\mathcal{K}$ is a simple group implies that $\mathcal{K}$ has to be the monster group $\textbf{M}$.

We will start with the family of lattices ${L(\tau)}={const\cdot{\eta^2(\tau)L_{\tau}}}$ on $\textit{H,}$ where $L_{\tau}=[1,\tau]$  and we will show that the moduli space for complex tori ${\textbf{T}(\tau)}\cong{\mathbb{C}/L(\tau)}$ is an elliptic open curve $\mathcal{E}: t^2=4u^3-1$. Using the ramification scheme for appropriate natural projections we obtain that $u(\tau)$ and $t(\tau)$ coicide with the absolute invariants for $\Gamma^{+}_{ns}(3)$ and for $\Gamma_c$ respectively as well as that the curve $\mathcal{E}$ is analytically isomorphic to ${\textbf{T}^{*}}\cong{\textit{H}/\Gamma'}\cong{\mathbb{C}-L_0/L_0}$.  In  section 3 we find  relations $dz_{\alpha}=s\eta^4(\tau_{\alpha})d\tau_{\alpha}$, (s is a global constant) between local coordinates on $\textbf{T}^{*}$ and we investigate their consequences. We introduce some Hecke operators and we find their images on some, important for us , automorphic functions and forms. The expression of the standard holomorphic quadratic differential $(dz_{\alpha}^2)_{\alpha}$ on  $\textbf{T}^{*}$ in terms of the Dedekind eta function allows us to find  the Bers embedding of $\mathcal{T}_{1,1}$ using the equation$(3.22)$ instead of  working with the Lame equation $(3.23)$.  In section 4 we investigate different realizations $(4.3)$ of the  quotient $\widetilde{\Gamma}/\Gamma'$, ($\widetilde{\Gamma}=PSL_{2}\mathbb{Z}$) that are naturally associated to the standard quadrilateral $\mathfrak{F}'_4$ and hexagonal $\mathfrak{F}'_6$ fundamental domains for $\Gamma'$ respectively. Only the latter one determines very important (although a non-unitary) representation of $\Gamma$ in the $2$-dimentional vector space spanned by the Weierstrass functions $\wp$ and $\wp'$ on $\textbf{T}^{*}$. In section 5 we construct the decomposition of the lattice ${{L_0}\cong{p(\Gamma'(\infty))}}$, (p is the natural projection $p:{\textit{H}\rightarrow{\mathbb{C}-L_0}}$) into eight disjoint subsets $\widetilde{\mathcal{L}}_k$. The symmetries of the lattice $L_0$ allow us to realize each $\widetilde{\mathcal{L}}_{k}$, $k=1,\ldots,8$,  as a sublattice of $L_0$ (which  is the 4-dilate of $L_0$) together with its appropriate half-points in three distinct ways. In section 6 we investigate conclusions of these decompositions and we find some sort of hidden $E_8$-symmetry on $\textbf{T}^{*}$.

\section{Preliminaries}
\subsection{Curve $\mathcal{E}$}
Each element $\tau$ of the upper half-plane $\textit{H}$ determines a lattice ${L_{\tau}}=[1,\tau]$ and a complex torus ${\textbf{T}_{\tau}}={\mathbb{C}/L_{\tau}}$.  However,  we will consider, instead of the standard family $\{\textbf{T}_{\tau}\}_{{\tau}\in{\textit{H}}}$ of compact complex tori,  a family $\{{\textbf{T}(\tau)}={\mathbb{C}}/L(\tau)\}_{{\tau}\in{\textit{H}}}$ where  $L(\tau)=\mu(\tau)L_{\tau}$, $\mu(\tau)=2\pi3^{-\frac{1}{4}}{\eta}^{2}(\tau)$ and  $\eta(\tau)$ is the standard Dedekind eta function. Now,  each torus $\textbf{T}(\tau)$ is analytically isomorphic to the curve 
\begin{equation}
E_{L(\tau)}:  Y^2 = 4X^3-g_{2}(L(\tau))X-g_{3}(L(\tau))
\end{equation} 
and we will define  a function $u(\tau)$ as given by $u(\tau):=\frac{1}{3\sqrt[3]{4}}g_{2}(L(\tau))$ and the function $t(\tau):=g_{3}(L(\tau))$. We have 
\begin{equation}
g_{2}(L(\tau))={\mu(\tau)}^{-4}g_{2}(\tau)={\frac{3}{(2\pi)^4}}{\frac{g_{2}(\tau)}{{\eta(\tau)}^8}} 
\end{equation}
and
\begin{equation}
g_{3}(L(\tau))={\mu(\tau)}^{-6}g_3(\tau)=\frac{3^{\frac{3}{2}}}{(2\pi)^6}{\frac{g_3(\tau)}{{\eta(\tau)}^{12}}}
\end{equation}
where $g_k(\tau)=g_k(L_{\tau})$ for $k=2,3$  are the standard Eisenstein series.  We see that the functions $u(\tau)$ and $t(\tau)$ satisfy the equation $4u^3-t^2-1=0$ and hence determine an elliptic open curve
\begin{equation}
\mathcal{E} :  t^2=4u^3-1
\end{equation}
Each point ${(u(\tau),t(\tau))}\in{\mathcal{E}}$ corresponds to a curve 
\begin{equation}
E_{u,t} :  Y^2 = 4X^3-3{\sqrt[3]{4}}uX-t
\end{equation}

When point $P=(u,t)$ of $\mathcal{E}$ has both coordinates different from zero then there exist exactly six distict points $({\rho}^{k}u,{\pm}{t})$  ($k=0,1,2$, $\rho=e^{\frac{2\pi{i}}{3}}$)  on $\mathcal{E}$ which correspond to six isomorphic elliptic curves representing the same equivalent class of complex tori. When $P=(u,0)$ then we must have $u=4^{-\frac{1}{3}}\rho^k$ and  points $({\rho}^{k}4^{-\frac{1}{3}},0)$ with $k=0,1,2$ correspond to three isomorphic curves representing the equivalence class $[\mathbb{C}/\mathbb{Z}[i]]$. When $P=(0,t)$ then we must have $t={\pm}i$ and both curves $E_{0,{\pm}i}$ represent the equivalence class $[{\mathbb{C}}/{\mathbb{Z}[\rho]}]$ of tori. (Here the square bracket denotes the equivalence class of complex tori i.e. a point of the modular space $\textit{H}/\Gamma$, $\Gamma=SL_{2}\mathbb{Z}$.) 

From the form of the equation $(2.4)$ the elliptic curve $\mathcal{E}$ is itself  analytically isomorphic to a complex torus that belongs to the class $[{\mathbb{C}}/{\mathbb{Z}[\rho]}]$. Since both functions $u(\tau)$ and $t(\tau)$ have the hyperbolic nature to find their realizations in terms of the Weierstrass functions $\wp$ and ${\wp}'$ (which belong to the flat geometry) we must consider the relationships between $\mathcal{E}$ and some modular curves of level 2 and of level 3 structures respectively.

\subsection{$\Gamma'$, $\Gamma_c$ and $\Gamma(2)$}.
Let $r_N$ denote the modulo N homomorphism $r_N$:  ${SL_{2}\mathbb{Z}}\rightarrow{SL_{2}(N)}$. The image ${r_{2}(\Gamma)}={SL_{2}(2)}\cong{S_3}$ whereas the image of ${\Gamma'}=[\Gamma,\Gamma]$ is the normal subgroup of $S_3$ given by ${\mathcal{C}_3}\cong{\mathbb{Z}_3}$. Let $\Gamma_c$ denote the subgroup ${r_{2}}^{-1}(\mathcal{C}_3)$ of $\Gamma$. It has genus zero, it has only one cusp of width 2 and it has index 2 in $\Gamma$.  Moreover  we may take $\{I,T\}$ as a set of its coset representatives in $\Gamma$, $T=\begin{pmatrix} 1&1\\ 0&1\end{pmatrix}$. Let $\textbf{T}^{*}$ denote the punctured torus $\textit{H}/{\Gamma}'$ which is analyticaly isomorphic to ${\mathbb{C}-L_0}/{L_0}$ for some lattice ${L_0}={const}\cdot{L_\rho}$ and let $\textbf{X}'$  be ${{\textbf{T}^{*}}\cup{\{\infty\}}}\cong{\textit{H}^{*}}/{\Gamma}'$ where $\textit{H}^{*}$ denotes the extended half-plane ${\textit{H}}\cup{\mathbb{Q}}\cup{\{\infty\}}$.  We have the following natural projections: ${\textbf{X}'}\stackrel{{{\pi}'}_c}{\rightarrow}{\textbf{X}_c}\stackrel{\pi_c}{\rightarrow}{\textbf{X}(1)}$ with ${\textbf{X}_c}\cong{{\textit{H}^{*}}/{\Gamma_c}}$, ${\textbf{X}(1)}\cong{{\textit{H}^{*}}/{\Gamma}}$  with projections ${\pi}'_c$ of degree 3 and $\pi_c$ of degree 2. The absolute invariant for $\Gamma_c$ is given by  ${\textit{J}_{c}(\tau)}=(\textit{J}(\tau)-1)^{\frac{1}{2}}$, ~\cite{JPS97},  and it   is also $\Gamma'$ invariant ( using $(2.3)$ it  may be identified with the function $t(\tau)$).  The comparison of the ramification scheme for ${{\pi}'_c}$: ${\textbf{X}'\rightarrow}{\textbf{X}_c}$ and for ${\wp}'$: ${\textbf{X}'}\rightarrow{\mathbb{C}P_1}$ implies that  (after the identification of $\mathbb{C}P_1$ with $\textit{J}_c$-plane $\textbf{X}_c$) the $\Gamma'$-automorphic function $t(\tau)$ coincides with the lifting to $\textit{H}$ of the function ${\wp}'$ on $\textbf{T}^{*}$. In other words we have shown that the following is true: 
\begin{lemma}
Let $p$: ${\textit{H}}\rightarrow{\textit{H}/N}$ be the natural projection coresponding to  the group  $N=[\Gamma',\Gamma']$ with ${\textit{H}/N}\cong{\mathbb{C}-L_0}$. The lifting of ${\wp}'(z,L_0)$ on $\mathbb{C}-L_0$ to $\textit{H}$ determined by $p$ produces exactly the $\Gamma'$-automorphic function $t(\tau)$.
\end{lemma}
At this point it is worth to notice that (since the  modulo 2 homomorphism maps  both groups  $\left\langle g\right\rangle$ and $\left\langle a\right\rangle$ onto $\mathcal{C}_3$ and since $a$ is $\Gamma(2)$-equivalent to $g^2$ and $a^2$ is $\Gamma(2)$-equivalent to $g$) we may view the modular curve $\textbf{X}_c$ (of $\mathcal{C}_3$-equivalent level two structures) as the quotient  ${\textbf{X}(2)}/{\mathcal{C}_3}$  (here $g=ST$, $a=TS$, $S=\begin{pmatrix}0&1\\-1&0\end{pmatrix}$ ).

\subsection{$\Gamma'$, $\Gamma(3)$ and $\Gamma^{+}_{ns}(3)$}                          A similar situation occurs when we pass to the modulo 3 homomorphism. The image $r_3(\Gamma')$  is the normal subgroup of ${\Gamma/{\Gamma(3)}}\cong{SL_2(3)}$ but now this subgroup is not an abelian one. It is isomorphic to the quaternion group $Q_8$ and we have 
\begin{equation}
\xymatrix{
1\ar[r] &Q_8\ar[r]& SL_2(3)\ar[r] & \mathbb{Z}_3\ar[r] &1\\
}
\end{equation}

The subgroup ${r_3}^{-1}(Q_8)$ of $\Gamma$ is associated to the non-split Cartan subgroup of $GL_2(3)$ and is usually denoted by $\Gamma^{+}_{ns}(3)$, ~\cite{ICH99}. It has index 3 in $\Gamma$ and we may take the set $\{I,T,T^2\}$ as a set of its coset representatives. The modular curve ${\textbf{X}^{+}_{ns}}(3)$ of $Q_8$-equivalent level 3 structures (in fact, since the normal subgroup $\mathcal{N}$ of $Q_8$ acts trivially, these structures are ${Q_8}/{\mathcal{N}}$ equivalent) has genus zero and only one cusp of width 3. 

The absolute invariant for $\Gamma^{+}_{ns}(3)$ can be taken as ${\textit{J}_n(\tau)}={\textit{J}(\tau)}^{\frac{1}{3}}$, ~\cite{BB09},  and hence this uniformizer of ${\textbf{X}^{+}_{ns}}(3)$  coincides with the $\Gamma'$-automorphic function  ${\sqrt[3]{4}}u(\tau)$ introduced earlier. Taking into account the ramification scheme given by Pict.1  
\vspace{12pt}
\begin{pspicture}(-0.5,-1.1)(10,3.5)
\rput(-0.1,3.2){\rnode{a}{$\textbf{X}'$}}
\rput(-0.1,1.7){\rnode{b}{${\textbf{X}^{+}_{ns}}(3)$}}
\rput(-0.1,0){\rnode{c}{$\textbf{X}(1)$}}
\rput(2,3.2){\rnode{f}{${\infty}'$}}
\rput(2,1.7){\rnode{e}{$\infty_1$}}
\rput(2,0){\rnode{d}{$\infty$}}
\rput(3,3.2){\rnode{i}{$\rho_1$}}
\rput(4,3.2){\rnode{j}{$\rho_2$}}
\rput(3.5,1.7){\rnode{h}{$\rho$}}
\rput(3.5,0){\rnode{g}{0}}
\rput(5,3.2){\rnode{k}{$i_1'$}}
\rput(5.5,3.2){\rnode{l}{$i_2'$}}
\rput(6,3.2){\rnode{m}{$i_3'$}}
\rput(5,1.7){\rnode{n}{$i_1$}}
\rput(5.5,1.7){\rnode{o}{$i_2$}}
\rput(6,1.7){\rnode{p}{$i_3$}}
\rput(5.5,0){\rnode{q}{1}}
\rput(7.5,3.2){\rnode{r}{$x_1'$}}
\rput(8,3.2){\rnode{s}{$x_1''$}}
\rput(8.5,3.2){\rnode{t}{$x_2'$}}
\rput(9,3.2){\rnode{u}{$x_2''$}}
\rput(9.5,3.2){\rnode{v}{$x_3'$}}
\rput(10,3.2){\rnode{z}{$x_3''$}}
\rput(7.7,1.7){\rnode{w}{$x_1$}}
\rput(8.7,1.7){\rnode{x}{$x_2$}}
\rput(9.7,1.7){\rnode{y}{$x_3$}}
\rput(8.7,0){\rnode{Q}{$x$}}
\psline(-0.1,3)(-0.1,1.9)
\psline(-0.1,1.5)(-0.1,0.2)
\psline(2,3)(2,1.9)
\psline(2,1.5)(2,0.2)
\psline(3,3)(3.3,1.9)
\psline(4,3)(3.7,1.9)
\psline(3.5,1.5)(3.5,0.2)
\psline(5,3)(5,1.9)
\psline(5.5,3)(5.5,1.9)
\psline(6,3)(6,1.9)
\psline(5,1.5)(5.4,0.2)
\psline(5.5,1.5)(5.5,0.2)
\psline(6,1.5)(5.6,0.2)
\psline(7.5,3)(7.6,1.9)
\psline(8,3)(7.8,1.9)
\psline(8.5,3)(8.6,1.9)
\psline(9,3)(8.8,1.9)
\psline(9.5,3)(9.6,1.9)
\psline(10,3)(9.8,1.9)
\psline(7.7,1.5)(8.6,0.2)
\psline(8.7,1.5)(8.7,0.2)
\psline(9.7,1.5)(8.8,0.2)
\rput(-0.4,2.6){\rnode{A}{2}}
\rput(-0.4,1.1){\rnode{B}{3}}
\rput(1.7,2.6){\rnode{C}{2}}
\rput(1.7,1.1){\rnode{D}{3}}
\rput(2.9,2.6){\rnode{E}{1}}
\rput(4.1,2.6){\rnode{F}{1}}
\rput(3.2,1.1){\rnode{G}{3}}
\rput(4.8,2.6){\rnode{H}{2}}
\rput(5.3,2.6){\rnode{I}{2}}
\rput(6.2,2.6){\rnode{J}{2}}
\rput(5,1.1){\rnode{K}{1}}
\rput(5.3,1.1){\rnode{L}{1}}
\rput(6,1.1){\rnode{M}{1}}
\rput(4.8,-1){\rnode{N}{$Pict.1$}}
\end{pspicture}

we obtain immediately:
\begin{lemma}
Let $p$ be the natural projection $\textit{H}\rightarrow{\mathbb{C}-L_0}$ introduced earlier. The lifting of the Weierstrass function $\wp(z,L_0)$ on $\mathbb{C}-L_0$ to $\textit{H}$ determined by $p$ produces exactly the $\Gamma'$-automorphic function $u(\tau)={\textit{J}_n(\tau)}$.
\end{lemma}

Since functions $u(\tau)$ and $t(\tau)$ are liftings to $\textit{H}$ of the Weierstrass functions $\wp$ and $\wp'$ respectively we have the following

\begin{corollary}
An elliptic curve $\mathcal{E}$: $t^2=4u^3-1$  that forms the moduli space of elliptic curves associated to the family of lattices $\{L(\tau)=\mu(\tau)L_{\tau}\}$ with $\tau\in{\textit{H}}$  is analytically isomorphic to the punctured torus $\textbf{T}^{*}={\textit{H}/\Gamma'}\cong{\mathbb{C}-L_0/L_0}$ with isomorphism  given by $z\rightarrow(u(\tau),t(\tau),1)$ for any  $\tau$ with the property  that $p(\tau)\in{z+L_0}$. 
\end{corollary}
Thus, the $\Gamma'$-automorphic functions $u(\tau)={\frac{1}{3\sqrt[3]{4}}}g_2(L(\tau))$ and $t(\tau)=g_3(L(\tau))$ are objects of both: of the euclidean geometry ( since $u(\tau)=\wp(p(\tau),L_0)$ and $t(\tau)=\wp'(p(\tau),L_0)$) and of the hyperbolic geometry (as $u(\tau)$ is the lifting to $\Gamma'$ of a Hauptmodule $\textit{J}_n(\tau)$ for $\Gamma^{+}_{ns}(3)$ and $t(\tau)$ is the lifting of a Hauptmodule $\textit{J}_c$ for $\Gamma_c$). In other words we have the following commutative diagrams:
\[
\xymatrix{
\textbf{X}'\ar[r]^{\wp'} \ar[d]_{\pi'_c}
&\mathbb{C}P_1\\
\textbf{X}_c\ar[ur]_{\textit{J}_c}\\
} \qquad    \xymatrix{
\textbf{X}'\ar[r]^{\sqrt[3]{4}\wp} \ar[d]_{\pi'_n}
&\mathbb{C}P_1\\
\textbf{X}^{+}_{ns}(3)\ar[ur]_{\textit{J}_n}\\
}
\]

\section{A Matter of the Dedekind Eta Function}
\subsection{Hyperbolic and Euclidean}
We have already introduced a universal covering $p$ which projects $\textit{H}$ onto the infinite punctured plane $\mathbb{C}-L_0$ with the deck group corresponding to a homomorphism of  ${\Pi_1(\mathbb{C}-L_0)}\rightarrow{N}$, $N=[\Gamma',\Gamma']$. So, ${N\tau}\Leftrightarrow{z\in\mathbb{C}-L_0}$, ${\Gamma'\tau}\Leftrightarrow{z+L_0}$ and ${L_0}=c[1,\rho]$ for some constant c. Let $r$ be the  local inverse of $p$, that is, $\{r,z\}=\frac{1}{2}\wp(z,L_0)$ (here \{\} denotes the Schwarzian derivative). Now the $\Gamma'$-automorphic functions $u$ and $t$ can be locally viewed as
\begin{equation}
u(r(z))=\wp(z,L_0) \qquad t(r(z))= \wp'(z,L_0)
\end{equation}

Let $\{(U_{\alpha},\tau_{\alpha}\}_{\alpha}$ be an atlas on ${\textbf{T}^{*}}\cong{\textit{H}/\Gamma'}$ coming from the universal covering $\pi'$: $\textit{H}\rightarrow\textbf{T}^{*}$ i.e. for $(u,t)=x'\in{U_{\alpha}}\cap{U_{\beta}}$,  we have $\tau_{\beta}(x')=\gamma\tau_{\alpha}(x')$ for some $\gamma\in\Gamma'$. Since  the multiplier system of $\eta^2(\tau)$ restricted to the subgroup $\Gamma'$ of $\Gamma$ is a trivial one, on any intersection ${U_{\alpha}}\cap{U_{\beta}}$  we obtain

\begin{equation}
L(\tau_{\beta})=\mu(\tau_{\beta})L_{\tau_{\beta}}=\mu(\tau_{\alpha})L_{\tau_{\alpha}}=L(\tau_{\alpha})
\end{equation}

This means that at each point $x'\in\textbf{T}^{*}$, $x'=(u,t)$, we have well define lattice $L(x')=L(\tau_{\alpha}(x'))=L(\tau_{\beta}(x'))$ and hence we have an analytic isomorphism between $\mathbb{C}/L(x')$ and $E_{u,t}$: $Y^2=4X^3-3\sqrt[3]{4}uX-t$.

Let us introduce another atlas $\{(U_{\alpha},z_{\alpha})\}_{\alpha}$ on $\textbf{T}^{*}$ with holomorphic bijections $z_{\alpha}$:$U_{\alpha}\rightarrow{V_{\alpha}}\subset{\mathbb{C}-L_0}$ coming from the projection $p': {\mathbb{C}-L_0}\rightarrow{\textbf{T}^{*}}$ and with the property that
\begin{equation}
\tau_{\alpha}(p'(z_{\alpha}))=r(z_{\alpha}) \qquad  z_{\alpha}(\pi'(\tau_{\alpha}))=p(\tau_{\alpha})
\end{equation}

(If necessary we may pass to some refinement of an open covering $\{U_{\alpha}\}$  of $\textbf{T}^{*}$.) Now, for each $x'\in{{U_{\alpha}}\cap{U_{\beta}}}$ we have ${\tau_{\beta}(x')}={\gamma\tau_{\alpha}(x')}$ for some ${\gamma}\in{\Gamma'}$ and ${z_{\beta}(x')}={z_{\alpha}(x')+w}$ for some $w\in{L_0}$. Since
\begin{equation*}
u(\tau_{\alpha})=\wp(p(\tau_{\alpha}),L_0)=\wp(p(\tau_{\beta}),L_0)=u(\tau_{\beta})
\end{equation*} 

and analogously
\begin{equation*}
t(\tau_{\alpha})=\wp'(z_{\alpha},L_0)=\wp'(z_{\beta},L_0)=t(\tau_{\beta})
\end{equation*}

the relation 
\begin{equation*}
t(\tau_{\alpha})=\wp'(z_{\alpha},L_0)= \frac{d\wp(p(\tau_{\alpha}),L_0)}{dp(\tau_{\alpha})}={\frac{d\wp(p(\tau_{\alpha}))}{d\tau_{\alpha}}}{\frac{d\tau_{\alpha}}{dz_{\alpha}}}={\frac{du(\tau_{\alpha})}{d\tau_{\alpha}}}{\frac{d\tau_{\alpha}}{dz_{\alpha}}}
\end{equation*}

implies that

\begin{equation}
\frac{du(\tau_{\alpha})}{d\tau_{\alpha}}=t(\tau_{\alpha}){\frac{dz_{\alpha}}{d\tau_{\alpha}}}
\end{equation}.

Since we already know that
\begin{equation*}
 u(\tau_{\alpha})={({\frac{1}{4}}\textit{J}(\tau_{\alpha}))^{\frac{1}{3}}}, \qquad  t(\tau_{\alpha})={(\textit{J}(\tau_{\alpha})-1)^{\frac{1}{2}}}
\end{equation*}
  we may use the well known formula ~\cite{KCH85}
  
\begin{equation*}
\eta^{24}(\tau)=\frac{1}{(48\pi^2)^3}{\frac{{\textit{J}'(\tau)}^6}{{(\textit{J}(\tau))^4}{(1-\textit{J}(\tau))^3}}}
\end{equation*}

to find that on ${U_{\alpha}}\subset\textbf{T}^{*}$ we have
\begin{equation}
\frac{dz_{\alpha}}{d\tau_{\alpha}}=s\eta^4(\tau_{\alpha}), \qquad s=2k\pi{\frac{\sqrt[3]{2}}{\sqrt{3}}}
\end{equation}

and  $k$ is a global constant given by a 6-th root of $-1$.  Hence, on any intersection  ${U_{\alpha}}\cap{U_{\beta}}$ on $\textbf{T}^{*}$ we have
\begin{equation}
dz_{\alpha}=s\eta^4(\tau_{\alpha})d\tau_{\alpha}=s\eta^4(\tau_{\beta})d\tau_{\beta}=dz_{\beta}
\end{equation}

as expected. We see that the Dedekind eta function provides the transition between the local flat coordinates $z_{\alpha}(x')$ and the hyperbolic $\tau_{\alpha}(x')$ coordinates on $\textbf{T}^{*}$. In other words  it plays the role of a bridge between the euclidean geometry on ${\textbf{T}^{*}}\cong{\textit{H}/\Gamma'}$ and its natural hyperbolic geometry. Moreover, from the formula $(3.6)$,  we obtain  that

\begin{equation}
 \wp(z_{\alpha},L_0)dz_{\alpha}^2=\frac{k}{3(2\pi)^3}g_2(\tau_{\alpha})d\tau_{\alpha}^2
\end{equation}

Let $q$ be the holomorphic quadratic differential  on $\textbf{T}^{*}$ that is determined by the Eisenstein series $g_2(\tau)$ i.e.  with respect to the atlas ${\{(U_{\alpha},\tau_{\alpha})\}}_{\alpha}$ it can be written  as $q={(\frac{k}{3(2\pi)}g_2(\tau_{\alpha})d\tau_{\alpha}^2)_{\alpha}}$.  Now, with respect to the atlas $\{(U_{\alpha},z_{\alpha})\}_{\alpha}$, $q$ takes the form:

\begin{equation}
q=(\wp(z_{\alpha},L_0)dz_{\alpha}^2)_{\alpha}
\end{equation}

Similarly, it is easy to check that 
\begin{equation}
t(\tau_{\alpha})= \frac{2\pi\sqrt{3}}{\sqrt[3]{2}}\frac{du(\tau_{\alpha})}{d\tau_{\alpha}}\eta^{-4}(\tau_{\alpha})
\end{equation}

and hence on each $U_{\alpha}$ we can write
\begin{equation}
g_3(\tau_{\alpha})= \frac{(2\pi)^7}{3\sqrt[3]{2}}e^{-\frac{i\pi}{6}}\eta^8(\tau_{\alpha})\frac{du(\tau_{\alpha})}{d\tau_{\alpha}}
\end{equation}

 Let $\xi$ denote the holomorphic differential on $\textbf{T}^{*}$ wich is determined by $g_3(\tau)$. The above formulae allow us to write
\begin{equation}
\xi=(\wp'(z_{\alpha},L_0)dz_{\alpha}^3)_{\alpha}=(\frac{2}{(2\pi)^9}e^{\frac{i\pi}{2}}g_3(\tau_{\alpha})d\tau_{\alpha}^3)_{\alpha}
\end{equation}

In other words we have shown the following:

\begin{lemma}
The $\Gamma'$-automorphic forms on $\textit{H}$ corresponding to the differentials $q=(\wp(z_{\alpha},L_0)dz_{\alpha}^2)_{\alpha}$ and $\xi=(\wp'(z_{\alpha},L_0)dz_{\alpha}^3)_{\alpha}$ on $\textbf{T}^{*}$ are exactly ones  determined by the standard Eisenstein series $g_2(\tau)$ and $g_3(\tau)$ respectively. More precisely we have
\begin{equation}
\frac{k^2}{3(2\pi)^2}g_{2}(\tau)=\frac{k^2\sqrt[3]{4}}{3}(2\pi)^{2}\eta^8(\tau)u(\tau)
\end{equation}
and 
\begin{equation}
\frac{2k^3}{(2\pi)^3}g_3(\tau)=\frac{k^3\sqrt[3]{2}}{3}(2\pi)^{2}\eta^8(\tau)\frac{du(\tau)}{d\tau}
\end{equation}
respectively.
\end{lemma}

From the relations $(3.12)$ and $(3.13)$ we obtain (after differentiating the first equation):
\begin{equation}
\frac{dg_2(\tau)}{d\tau}=8\frac{\eta'(\tau)}{\eta(\tau)}g_2(\tau)+\frac{3k}{\pi}g_3(\tau)
\end{equation}

We notice that when we choose the $6$-th root $k$ of $-1$ as $k=-i$ then the latter formula is equivalent to the Serre derivative of the modular form $E_4(\tau)=\frac{3}{2\pi^2}g_2(\tau)$.

\subsection{Some Hecke Operators}
Let us introduce (see ~\cite{ABV90}) the operator $T_{{\left\langle g\right\rangle},k}$ of weight $k\in{\mathbb{Z}}$ acting on the space of functions $f:{\textit{H}}\rightarrow{\mathbb{C}}$ as follows
\begin{equation}
(T_{{\left\langle g\right\rangle},k}f)(\tau)=\sum_{r=1}^3{j_{g^r}}^{-k}f(g^{r}\tau)
\end{equation}

where $j_{\gamma}(\tau)=c\tau+d$ for any element $\gamma=\begin{pmatrix}a&b\\ c&d\end{pmatrix}$ in $SL_2{\mathbb{R}}$. Let $\mathcal{A}_k(G)$ denote the space of $G$-automorphic forms of weight $k$ for a Fuchsian group $G$. Since we have 
\begin{equation}
T_{{\left\langle g\right\rangle},k}:{\mathcal{A}_k(\Gamma')}\rightarrow{\mathcal{A}_k(\Gamma_c)}
\end{equation}

and
\begin{equation}
T_{{\left\langle g\right\rangle},k}: {\mathcal{A}_k(\Gamma(2))}\rightarrow{\mathcal{A}_k(\Gamma_c)}
\end{equation}

we may  find some relations between $k$-forms for $\Gamma(2)$ and $k$-forms for $\Gamma'$. Namely, these Hecke operators together with the projections $\pi'_c$ and ${\pi^c_2}:{X(2)\rightarrow{X_c}}$  allow us to transform  $\frac{k}{2}$-differentials on $\textbf{X}'$ into $\frac{k}{2}$-differentials on $\textbf{X}(2)$ and vice versa. Let us denote the composition of $(\pi^c_2)^{*}$ and of $T_{{\left\langle g\right\rangle},k}$ as the operator $\widehat{H}_k$. Let us check what are the images of $\widehat{H}_0$ produced by the $\Gamma'$-automorphic functions $u(\tau)$ and $t(\tau)$.  Since $L(g\tau)=ie^{\frac{i\pi}{6}}L(\tau)$ and $L(g^{2}\tau)=-ie^{-\frac{i\pi}{6}}L(\tau)$ we have ${g_2(L(g\tau))}={\rho^{2}g_2(L(\tau))}$ and ${g_2(L(g^{2}\tau))}={\rho}g_2(L(\tau))$. Hence 

\begin{equation*}
 \widehat{H}_{0} u(\tau)=0 \qquad and \qquad  \widehat{H}_{0}t(\tau)=3t(\tau)
\end{equation*} 

So, the Weierstrass function $\wp$ on $\textbf{T}^{*}$ produces the zero function on $\textbf{X}(2)$  but the Weierstrass function $\wp'$ produces a multiple of the lifting of the  absolute invariant $\textit{J}_c$ from $\textbf{X}_c$ to $\textbf{X}(2)$.  We already know  that the regular quadratic differential $(dz_{\alpha}^2)_{\alpha}$ on $\textbf{T}^{*}$ corresponds to the  $\Gamma'$-automorphic form $s^{2}\eta^8(\tau)$ on $\textit{H}$ . It occurs that the image under the operator $\widehat{H}_4 = (\pi^c_2)^{*}\circ{T_{{\left\langle g\right\rangle},4}}$ (transforming $\mathcal{A}_4(\Gamma')$ into $\mathcal{A}_4(\Gamma(2))$)  of $\eta^8(\tau)$ vanishes. Although $\widehat{H}_{0}u(\tau)=0$ and $\widehat{H}_{4}\eta^8(\tau)$=0 the operator $\widehat{H}_4$ acts on their product $u(\tau)\eta^8(\tau)$ by multiplication by $3$. This is because the product is a $\Gamma$-automorphic form i.e. ${u(\tau)\eta^8(\tau)}\in{\mathcal{A}_4(\Gamma)}\subset{\mathcal{A}_4(\Gamma')}$.  Generally we have
\begin{lemma}
For any ${\varphi}\in{\mathcal{A}_k(\Gamma)}$ and for any $f\in{\mathcal{A}_0(\Gamma')}$ we have ${\widehat{H}_k(f\varphi)}={\varphi\widehat{H}_0(f)}$.
\end{lemma} 
\begin{proof}
Simple
\end{proof}

We have exactly the same properties when we replace $\mathcal{A}_k(\Gamma')$ by $\mathcal{A}_k(\Gamma(2))$ and the operator $\widehat{H}_k$ by the the operator $\widetilde{H}_k$ defined as the composition ${{\pi'_c}^{*}}\circ{T_{\left\langle g\right\rangle,k}}$ and transforming $\mathcal{A}_k(\Gamma(2))$ into $\mathcal{A}_k(\Gamma')$. Since we have
\begin{equation}
{T_{\left\langle g\right\rangle,4}\vartheta_3(\tau)^8}=\frac{3}{(2\pi)^4}g_2(\tau)
\end{equation} 
the image by $\widetilde{H}_4$ of the differential on $\textbf{X}(2)$ determined by ${\vartheta_3(\tau)^8}\in{\mathcal{A}_4(\Gamma(2))}$ produces the differential 
\begin{equation}
{(\frac{1}{k^2}(\frac{3}{2\pi})^{2}\wp(z_{\alpha},L_0)dz_{\alpha}^2)_{\alpha}}={(\frac{3}{(2\pi)^4}g_2(\tau_{\alpha})d\tau_{\alpha}^2)_{\alpha}}
\end{equation}
on $\textbf{T}^{*}$. (Here ${\vartheta_3(\tau)}\equiv{\vartheta_3(0,\tau)}$ is the standard theta function on $\textit{H}$.) However when we start with $g_2(\tau)$ as $\Gamma'$-automorpfic form  then, using the operatoe $\widehat{H}_4$ we will not return to ${\vartheta_3(\tau)^8}\in{\mathcal{A}_4(\Gamma(2))}$. Instead of we obtain ${\widehat{H}_{4}g_2(\tau)}=3g_2(\tau)$ as an element of $\mathcal{A}_4(\Gamma(2))$.

\subsection{Bers embedding}
Since the differential $q$ given by $(3.8)$ has a pole of order $2$ at the puncture of $\textbf{T}^{*}$ it is not integrable so, although it is holomorphic on $\textbf{T}^{*}$, it does not correspond to any element of the Banach space $\mathfrak{B}_2(\textit{L},\Gamma')$ (the space of all holomorphic Nehari-bounded forms on the lower half-plane $\textit{L}$ of weight 4, ~\cite{SNAG88}). This means that we cannot use $q$ to construct the Bers embedding $\mathcal{T}_{1,1}\rightarrow{\mathfrak{B}_2(\textit{L},\Gamma')}$. However, we see from $(3.6)$  that the holomorphic differential on $\textbf{T}^{*}$

\begin{equation}
\varphi=(dz_{\alpha}^2)_{\alpha}=(s^{2}\eta^8(\tau_{\alpha})d\tau_{\alpha}^2)_{\alpha}
\end{equation}

corresponds to  $\Phi=\phi(\tau)d\tau^2$  with   $\varphi(\tau)={s^{2}\eta^8(\tau)}$ and hence it corresponds to   an element of $\mathfrak{B}_2(\textit{L},\Gamma')$ which  may be  used  to find a concrete Bers embedding. However now, the space  $\mathcal{T}_{1,1}$  must have its  origin at $\textbf{T}^{*}$.  This means that we must find the domain of complex numbers $b$ such that the Schwarzian differential equation
\begin{equation}
\{w,\tau\}=b\varphi(\tau)
\end{equation}

has a schlicht solution $w$ which has a quasiconformal extention $\widehat{w}$ to all $\mathbb{C}$ compatible with $\Gamma'$. Since a shlicht solution $w$ of $(3.21)$ can be written as the quotient $\frac{y_1}{y_2}$ of two linearly independent solutions of $y''(\tau)+\frac{1}{2}b\varphi(\tau)y(\tau)=0$ to find the values of $b$ for which $\widehat{w}\Gamma'{\widehat{w}^{-1}}$ is a quasi-Fuchsian of signature $(1;1)$ we should consider the linear differential equation

\begin{equation}
y''(\tau)+\frac{bs^2}{2}\eta^8(\tau)y(\tau)=0 \qquad \tau\in{\textit{L}}
\end{equation}

 Till now, to find a Bers embedding of $\mathcal{T}_{1,1}$ we take the the Teichmueller space $\mathcal{T}_{1,1}$ originating at the punctured torus, usually ${\mathbb{C}-L_i}/L_i$,  and we are looking for the values $b\in\mathbb{C}$ for which the Lame equation 
 
\begin{equation}
y''+ \frac{1}{2}(\frac{1}{2}\wp(z,L_i)+b)y=0
\end{equation}

has a purely parabolic monodromy group (which is the commutator subgroup of the quasi-Fuchsian group $\widehat{w}\Gamma_i{\widehat{w}^{-1}}$ of signature $(1;1)$) Thus, the relation $(3.6)$ allows us to consider the equation $(3.22)$ instead of $(3.23)$ and, since till now the equation $(3.22)$ had not been investigated (to the author's knowledge), there is a possibility that we obtain new, more transparent understanding of the domain of Bers embedding of the Teichmueller space $\mathcal{T}_{1,1}$ that originates at $\textbf{T}^{*}={\textit{H}/\Gamma'}$ (instead of at $\textbf{T}^{*}_i$)

\section{Fundamental Domains for $\Gamma'$}
The standard quadrilateral fundamental domains $\mathfrak{F}'_4$ for $\Gamma'$ and $\mathfrak{F}(\Gamma(2))$ for $\Gamma(2)$ have the same underlying set $\mathfrak{F}$, given by the quadrilateral $(-1,0,1,\infty)$, and hence we may choose the same set of their coset representatives in ${\widetilde{\Gamma}}={PSL_{2}\mathbb{Z}}$.  We may decompose the set $\mathfrak{F}$ into copies of a fundamental region  ${F(\Gamma)}=(i-1,\rho,i,\infty)$ or into copies of a fundamental region ${F_\Gamma}={(0,\rho+1,\infty)}$ of the modular group according to
\begin{equation}
\mathfrak{F}={\mathfrak{S}_{1}F(\Gamma)}={\mathfrak{S}_{2}F_{\Gamma}}
\end{equation}
where ${\mathfrak{S}_1}={\{I,g,g^2,T,Tg,Tg^2\}}$ and ${\mathfrak{S}_2}={\{I,a,a^2,S,Sa,Sa^2\}}$ are two sets of coset representatives in $\widetilde{\Gamma}$. When we start with the set $\mathfrak{F}$, to determine whether we have $\Gamma(2)$ or $\Gamma'$ quadrilateral domain, we have to use either geometric or algebraic considerations. Geometrically, we have different identifications on the border $\partial\mathfrak{F}$ given by the generators of ${\Gamma'}=\left\langle A,B\right\rangle$ and of ${\Gamma(2)}=\left\langle -I,T^2,U\right\rangle$ respectively.

Algebraically, the free generators $S$ and $g$ of ${\widetilde{\Gamma}}={\left\langle S\right\rangle}\ast{\left\langle g\right\rangle}$ determine distinct permutations of the copies of fundamental domains for $\widetilde{\Gamma}$ depending whether their union forms $\mathfrak{F}(\Gamma(2))$ or $\mathfrak{F}'_4$. More precisely, following the Millington construction ~\cite{AS02}, both $S$ and $g$ determine permutations $\mu$ and $\sigma$ of a set of coset representatives. A permutation group $\Sigma=\left\langle \mu,\sigma| \mu^2=\sigma^3=I\right\rangle$  acts transitively on a set of cosets and the disjoint cycle decomposition of $\mu$, $\sigma$ and of their product $\mu\sigma$ provides the genus and inequivalent cusp widths for an appropriate subgroup of $\widetilde{\Gamma}$. 

 For example, if we consider cosets represented by elements of $\mathfrak{S}_1$ and if we denumerate its elements as ${\{I,T,g^2,Tg^2,g,Tg\}}\Leftrightarrow{\{0,1,2,3,4,5\}}$ respectively then, the permutation $\mu={(03)(14)(25)}\in{S_6}$  for the both subgroups $\Gamma(2)$ and  $\Gamma'$ of $\widetilde{\Gamma}$. However the motion $g$ produces the permutation ${\sigma'}=(042)(153)$ for $\Gamma'$ and hence $\mu\sigma'=(0,1,2,3,4,5)$. The corresponding permutation group ${\Sigma(\Gamma')}={\left\langle \mu,\sigma'\right\rangle}$ tells us that $\Gamma'$ has genus $1$, no elliptic elements and the single cusp of width $6$ (equal to the lenght of the cycle $\mu\sigma'$). For $\Gamma(2)$, $g=ST$ generates  the permutation ${\sigma}=(042)(135)$. So the product $\mu\sigma=(01)(23)(45)$ and we have three inequivalent cusps of width equal to $2$ each.
 
 We notice that the cycle structures of the generators of $\Sigma(\Gamma')$ and of $\Sigma(\Gamma(2))$ are the same but the permutations given by the products of the generators introduce distinction in the properties of cusps for $\Gamma'$ and for $\Gamma(2)$ respectively. Of course we could take different enumeration of cosets and different decompositions of the fundamental region of a given subgroup of $\widetilde{\Gamma}$. The permutation group obtained by using these new data will have generators (i.e.permutations representing $S$ and $g$) that are simultaneously conjugate in $S_6$ either to $\{\mu,\sigma'\}$ (in the case of $\Gamma'$) or to $\{\mu,\sigma\}$ (for $\Gamma(2)$).
 
 Thus, when we start with the quadrilateral region $\mathfrak{F}=(-1,0,1,\infty)$ then we must perform some operations for the cusp of width 6 of $\Gamma'$  to be seen. However, the hexagonal fundamental domain ${\mathfrak{F}'_6}=(\rho-2,\rho-1,\rho,\omega,\omega+1,\omega+2,\omega+3,\infty)$ has the parabolic vertex of index 6 already. Moreover, if we choose the following fundamental standard domains: $R=(\rho,\omega,\infty)$ for $\widetilde{\Gamma}$, $F(\Gamma_c)={T^{-2}R\cup{T^{-1}R}}$ for $\Gamma_c$  and ${F(\Gamma^{+}_{ns}(3))}={T^{-2}R\cup{T^{-1}R}\cup{R}}$  for $\Gamma^{+}_{ns}(3)$ then we have immediately the relations between appropriate sets given by
 \begin{equation}
 {\mathfrak{F}'_6}=(I\cup{T^3})F(\Gamma^{+}_{ns}(3))=(I\cup{T^2}\cup{T^4})F(\Gamma_c)
 \end{equation}
 
 These relations immediately describe the ramifications of ${\textbf{X}'}\rightarrow{\textbf{X}^{+}_{ns}(3)}$ and of $\textbf{X}'\rightarrow{\textbf{X}_c}$ at $\infty$  respectively.

When we work with the quadrilateral domain $\mathfrak{F}'_4$ then, in fact, we are dealing with the subgroups of $PSL_{2}\mathbb{Z}/\Gamma'$ that may be identified with the finite subgroups $\left\langle S\right\rangle$ and $\left\langle g\right\rangle$ of the modular group itself. But when we consider the hexagonal fundamental domain $\mathfrak{F}'_6$ then the more natural is to view the quotient $PSL_{2}\mathbb{Z}/\Gamma'$ as given by $\left\langle T\right\rangle$mod$T^6$. Although  we have that $T^3$ is equivalent to $S$ modulo $\Gamma'$ (more precisely $T^{3}=S[S^{-1},T][S^{-1},T^{-1}]$) and we can write
\begin{equation}
{\widetilde{\Gamma}/\Gamma'}\cong{\left\langle S\right\rangle}\times{\left\langle g\right\rangle}\cong{\left\langle T^3\right\rangle}\times{\left\langle T^2\right\rangle}modT^6
\end{equation}

we notice that the elements $S$ and $g$ have finite order in $\widetilde{\Gamma}$ whereas both $T^3$ and $T^2$ are generators of infinite parabolic subgroups of $PSL_{2}\mathbb{Z}$. So we are dealing with transparent differences between the nature of algebraic objects that may be associated to $\mathfrak{F}'_4$ and $\mathfrak{F}'_6$ respectively and which are involved in the hidden structure  of the Veech curve determined by the dynamical system of a  billiard (in a $(\frac{\pi}{2},\frac{\pi}{3},\frac{\pi}{6})$-triangle) and described by the error correcting Golay code $\textsl{G}_{24}$ in ~\cite{KMA10}. These differences become even  deeper when we consider a non-unitary representation $\chi$ of $\widetilde{\Gamma}$ in $\mathbb{C}^{2}=Span\{{\textit{J}_c},{\textit{J}_n}\}$.  Since the $\Gamma'$-automorphic functions $u(\tau)$ and $t(\tau)$ are given by the liftings of $\textit{J}_n$ and of $\textit{J}_c$ from $\Gamma^{+}_{ns}(3)$ and from $\Gamma_c$ to $\Gamma'$ appropriately we may identify the underlying  vector space $\mathbb{C}^2$ for $\chi$ with the linear
 span of the Weierstrass functions $\wp(p(\tau),L_0)\cong{u(\tau)}$ and $\wp'(p(\tau),L_0)\cong{t(\tau)}$ respectively. Since
 \begin{equation*}
 u(\frac{-1}{\tau})=u(\tau) \qquad u(\tau+1)=\rho{u(\tau)}
 \end{equation*}
 and 
 \begin{equation*}
 t(\frac{-1}{\tau})=t(\tau)  \qquad t(\tau+1)=-t(\tau)
 \end{equation*}
 the transformation $S$ acts as identity.  We have:  $\chi(S)=I$, $\chi(T)=\begin{pmatrix}-1&0\\ 0&\rho\end{pmatrix}$ and $\chi(T^6)=I$. Thus to see $\chi$ as a representation of $\widetilde{\Gamma}/\Gamma'$ on $Span\{\wp,\wp'\}$ we must take the set $\{T^k,k=0\ldots5\}$ as a set of the cosets representatives of $\Gamma'$ in $\widetilde{\Gamma}$. In other words, it is the cusp of $\Gamma'$ and its ramification indices over $\textbf{X}_c$ and over $\textbf{X}^{+}_{ns}(3)$ respectively that are important here,  and  it is the hexagonal domain $\mathfrak{F}'_6$ which immediately produces the relations $(4.2)$.

\section{Decomposition of $L_0$}
The projection $p:\textit{H}\rightarrow{\textit{H}/N}\cong{\mathbb{C}-L_0}$, $N=[\Gamma',\Gamma']$ corresponds to the abelization of $\Gamma'$, ${\Gamma'/N}\cong{\mathbb{Z}^2}$. More precisely, let $\Gamma'$ be generated by $A=[S,T^{-1}]=\begin{pmatrix}1&1\\1&2\end{pmatrix}$ and by $B=[S,T]=\begin{pmatrix}1&-1\\ -1&2\end{pmatrix}$. Any element $\gamma\in{\Gamma'}$ has the abelianized form
\begin{equation}
\gamma=A^{m}B^{n}\mathfrak{n},\qquad (m,n)\in{\mathbb{Z}^2}, \qquad {\mathfrak{n}}\in{N}
\end{equation}
We usually write $\gamma=mA+nB$, ~\cite{HC81}, so that
\begin{equation*}
p(\mathfrak{n}\tau)=p(\tau)=z\in{\mathbb{C}-L_0}, \qquad L_0=[\omega_1,\omega_2]=c[1,\rho]
\end{equation*}
and 
\begin{equation*}
p(\gamma\tau)=p(\tau)+m\omega_1+n\omega_2 \qquad for\qquad \gamma=A^{m}B^{n}\mathfrak{n}\in{\Gamma'}
\end{equation*}

Let the quaternion group $Q_8=\left\langle \alpha,\beta |\alpha^4=1,\alpha^2=\beta^2,\alpha\beta=\beta\alpha^{-1}\right\rangle$ be realized by the following matrices in $SL_{2}(3)$:
\begin{equation}
I=\begin{pmatrix}1&0\\0&1\end{pmatrix},\quad \alpha=\begin{pmatrix}0&2\\1&0\end{pmatrix},\quad \alpha^2=\begin{pmatrix}2&0\\0&2\end{pmatrix}, \quad \alpha^3=\begin{pmatrix}0&1\\2&0\end{pmatrix},\quad
\end{equation}
\begin{equation*}                               \beta=\begin{pmatrix}1&1\\1&2\end{pmatrix},\quad \beta^3=\begin{pmatrix}2&2\\2&1\end{pmatrix},\quad \beta\alpha=\begin{pmatrix}1&2\\2&2\end{pmatrix},\quad \alpha\beta=\begin{pmatrix}2&1\\1&1\end{pmatrix}
\end{equation*}

The group $\mathbb{Z}_3$ that occurs in $(2.6)$ is generated by $X=\begin{pmatrix}1&1\\0&1\end{pmatrix}$ and  acts on $Q_8$ by the automorphisms determined by $X{\alpha}X^{-1}={\beta}$ and  $X{\beta}X^{-1}={\alpha\beta}$. Let $r'_3$ denote the restriction of the homomorphism $r_3$ to $\Gamma'$. It maps
\begin{equation}
A\rightarrow\beta,\quad A^2\rightarrow\beta^2,\quad A^3\rightarrow\beta^3,\quad A^4\rightarrow{I}
\end{equation}
\begin{equation*}
B\rightarrow{\beta\alpha},\quad AB\rightarrow\alpha^3, \quad A^{2}B\rightarrow{\alpha\beta}, \quad A^{3}B\rightarrow\alpha
\end{equation*}

From now on we will use the following enumeration of the elements of $Q_8$:
\begin{equation}
\{I,\beta,\beta^2,\beta^3,\beta\alpha,\alpha^3,\alpha\beta,\alpha\}\equiv\{q_1,q_2,\ldots,q_8\}
\end{equation}
respectively. Let $\sigma\in{S_8}$ be the permutation $\sigma=(13)(24)(57)(68)$ of $\{q_1,\ldots,q_8\}$ corresponding to the  multiplication by $\alpha^2=\beta^2=(\alpha\beta)^2=-I$.
\begin{lemma}
The homomorphism $r'_{3}:\Gamma'\rightarrow{Q_8}$ induces a unique mapping $\kappa:{\Gamma'/N}\rightarrow{Q_{8}}\times{Q_{8}}$ such that $\kappa(m,n)=(q_k,{\sigma}q_k)$ for some unique, appropriate $k\in\{1,\ldots,8\}$
\end{lemma}
\begin{proof}
Let $\mathcal{N}=\{1,\alpha^2\}$ denote a normal subgroup of $Q_8$ and let $r'_{3,N}$ denote the restriction of $r'_3$ to the normal subgroup $N$ of $\Gamma'$. Let $K_N$ denote the kernel of the homomorphism $r'_{3,N}$, $K_N\triangleleft{N}$, so that we have $N={K_N}\cup{\mathcal{A}}$ as a set, with $\mathcal{A}={{r'_{3,N}}^{-1}(\alpha^2)}$.    Since each coset $(m,n)$ of $N$ has the decomposition
\begin{equation}
(m,n)\equiv{A^{m}B^{n}N}={A^{m}B^{n}K_N}\cup{A^{m}B^{n}\mathcal{A}}
\end{equation}
and all elements of the set $\{A^{m}B^{n}K_N\}$ are mapped onto some concrete $q_k$ whereas elements of $\{A^{m}B^{n}\mathcal{A}\}$ are all mapped onto $\sigma{q_k}$, the lemma follows.
\end{proof}

For each $k\in\{1,\ldots,8\}$ we introduce the subset $\mathcal{A}_k$ of $\Gamma'$ as the union ${\mathcal{A}_k}={\mathfrak{A}_k}\cup{\mathfrak{B}_k}$ with
\begin{equation}
\mathfrak{A}_k=\{A^{m}B^{n}\mathfrak{n} | \mathfrak{n}\in{K_N}; r'_3(A^{m}B^{n})=q_k\}
\end{equation}
and with
\begin{equation}
\mathfrak{B}_k=\{A^{m'}B^{n'}\mathfrak{n}| \mathfrak{n}\in{\mathcal{A}}; r'_3(A^{m'}B^{n'})=\sigma{q_k}\}
\end{equation}.
\begin{lemma}
The decomposition $\Gamma'=\bigcup_{k=1}^{8}\mathcal{A}_k$ determines a one-one correspondence between the set of elements of $Q_8$ and the set of elements of  ${\mathbb{Z}_4}\times{\mathbb{Z}_2}$
\end{lemma}
\begin{proof}
Let us write 
\begin{equation}
\mathfrak{A}_k=\{A^{m_k}B^{n_k}\}K_N \quad and\quad \mathfrak{B}_k=\{A^{m'_k}B^{n'_k}\}{\mathcal{A}}
\end{equation}
 for appropriate pairs of integers $(m_k,n_k)$ and $(m'_k,n'_k)$ in $\mathbb{Z}^2$. Let $s_k$ and $t_k$ denote the smallest nonnegative integers such that $r'_3(A^{s_k}B^{t_k})=q_k$. We see immediately that
 \begin{equation}
 \{(m_k,n_k)\}=\{(4m+s_k,4n+t_k),(4m+2+s_k,4n+2+t_k), m,n\in{\mathbb{Z}}\}
 \end{equation}
 and
 \begin{equation}
 \{(m'_k,n'_k)\}=\{(4m+2+s_k,4n+t_k),(4m+s_k,4n+2+t_k), m,n\in{\mathbb{Z}}\}
 \end{equation}
Thus the set ${\mathcal{A}_k}\subset{\Gamma'}$ is uniquely determined by the pair $(s_k,t_k)\in{\mathbb{Z}_4}\times{\mathbb{Z}_2}$ and the lemma follows.
\end{proof}

The explicit relations between $Q_8$ and ${\mathbb{Z}_4}\times{\mathbb{Z}_2}$ are given in the table.
\begin{table}[t]
\begin{center}
\begin{tabular}{|c||c|c|c|c|c|c|c|c|}
\hline
$q_k$&$q_1$&$q_2$&$q_3$&$q_4$&$q_5$&$q_6$&$q_7$&$q_8$\\
\hline
$a_k=(s_k,t_k)$&(0,0)&(1,0)&(2,0)&(3,0)&(0,1)&(1,1)&(2,1)&(3,1)\\
\hline
\end{tabular}
\end{center}
\end{table}

We recall that the lattice $L_0$ is produced by the images of ${\infty}\in{\textit{H}^{*}}$ under $\Gamma'$,  ${L_0}\cong{p(\Gamma'(\infty))}$, and that it is identified with the quotient ${\Gamma'/N}\cong{\mathbb{Z}^2}$. Our previous  considerations lead us to the following:
\begin{lemma}
The homomorphism $r'_3:\Gamma'\rightarrow{Q_8}$ determines the decomposition of the lattice $L_0$ into $8$ disjoint sublattices.
\end{lemma}
\begin{proof}
We have seen that we can decompose the set of all $N$-cosets in $\Gamma'$
 into $8$ subsets of cosets given by $\{(m_k,n_k)\in{\mathbb{Z}^2} | A^{m_k}B^{n_k}\stackrel{r'_3}{\rightarrow}{q_k}\}$  i.e. produced by  cosets representatives ${A^{m_k}B^{n_k}}\in{\mathfrak{A}_k}$. Now, to each  such subset we may uniquely associate the subset ${\mathcal{L}_k}\subset{\mathbb{Z}^2}$ of the form:
 \begin{equation}
 \mathcal{L}_k={\{a_k+4\mathbb{Z}^2\}}\cup{\{a_k+(2,2)+4\mathbb{Z}^2\}}
 \end{equation}
 Using the correspondence $(m,n)\Leftrightarrow{m\omega_1+n\omega_2}$ as well as the symmetry properties of the lattice $L_0=[\omega_1,\omega_2]=\omega_1[1,\rho]$ expressed by $[1,\rho]=[1,\omega]$ for $\omega={\rho+1}=e^{\frac{i\pi}{3}}$ we obtain immediately that each subset ${\mathcal{L}_k}\subset{\mathbb{Z}^2}$, $k\in{\{1,\ldots,8\}}$ determines a unique sublattice $\widetilde{\mathcal{L}_k}$ of $L_0$ given by
 \begin{equation}
 \widetilde{\mathcal{L}_k}={\widetilde{a}_k+\omega_1[4,2\omega]}\subset{L_0}, \qquad \widetilde{a}_k=s_{k}\omega_1+t_{k}\omega_2
 \end{equation} 
\end{proof}

We notice that the $K$-multiple (for $K=\frac{\pi}{2}\vartheta_3^2(0,\omega)$) of the lattice $[4,2\omega]$  gives the primitive periods of the function sinus amplitudis $sn(2Kz)$.  We will not pursuit this direction here. Instead of we will look at $\widetilde{\mathcal{L}_k}$ as the sublattice $L_k={\widetilde{a}_k+4L_0}$ of $L_0$  together with its halfpoints $\{h_k\}={\widetilde{a}_k+2(\omega_1,\omega_2)+4L_0}$. We see immediately that although the decomposition of $L_0$ into mutually disjoint subsets $\widetilde{\mathcal{L}_k}$, $k=\{1,\ldots,8\}$ is uniquely determined by $q_k$'s, the realization of each $\widetilde{\mathcal{L}_k}$ as a sublattice of $L_0$ (more precisely a $4$-dilate of $L_0$) together with appropriate half points is not a canonical one.  Namely we can do this in three distinct ways. To analyse the situation let us introduce the lattices:
\begin{gather}
\Lambda_3=[\omega_1^3,\omega_2^3]=4[\omega_1,\omega_2]\\
\Lambda_4=[\omega_1^4,\omega_2^4]=\begin{pmatrix}0&1\\-1&-1\end{pmatrix}\circ{\Lambda_3}=[\omega_2^3,-\omega_1^3-\omega_2^3]\\
\Lambda_2=[\omega_1^2,\omega_2^2]=\begin{pmatrix}-1&-1\\1&0\end{pmatrix}\circ{\Lambda_3}=[-\omega_1^3-\omega_2^3,\omega_1^3]
\end{gather}
Equivalently we could write $\Lambda_4=g\circ{\Lambda_3}$ and $\Lambda_2=g^2\circ{\Lambda_3}$, $g=ST$.  Although all these three lattices $\Lambda_i$'s, $i=2,3,4$ are equivalent, the realizations of each subsets $\widetilde{\mathcal{L}_k}$ by distinct $\Lambda_i$'s requires distinct half points of these lattices. Thus we have:
\begin{equation}
I.\quad  \widetilde{\mathcal{L}_k}={\{\widetilde{a}_k+\Lambda_3\}}\cup{\{\widetilde{a}_k+\frac{\omega_1^3+\omega_2^3}{2}+\Lambda_3\}}
\end{equation}
or
\begin{equation}
II.\quad \widetilde{L}_k={\{\widetilde{a}_k+\Lambda_4\}}\cup{\{\widetilde{a}_k+\frac{\omega_2^4}{2}+\Lambda_4\}}
\end{equation}
or
\begin{equation}
III.\quad \widetilde{\mathcal{L}_k}={\{\widetilde{a}_k+\Lambda_2\}}\cup{\{\widetilde{a}_k+\frac{\omega_1^2}{2}+\Lambda_2\}}
\end{equation}
We observe that the realizations $I,II$ and $III$ are associated  to the pairs $(Q_8,I)$,  $(Q_8,g)$ and to $(Q_8,g^2)$ respectively.

\section{Conclusions}
We have seen that we needed both homomorphisms, modulo 2 and modulo 3, to find that the Weierstrass functions $\wp$ and $\wp'$ on $p(\textit{H})\cong{\mathbb{C}-L_0}$ have liftings to $\textit{H}$  given by the absolute invariants ${\textit{J}_n(\tau)}={(\textit{J}(\tau))^{\frac{1}{3}}}$ (for $\Gamma^{+}_{ns}(3)$) and ${\textit{J}_c(\tau)}={(\textit{J}(\tau)-1)^{\frac{1}{2}}}$, (for $\Gamma_c$) respectively. Further, we have obtained that the homomorphism $r'_3$ determines  the decomposition of $\Gamma'$ into subsets ${\mathcal{A}_k}={{r'}_{3}^{-1}(q_k)}$,  $k=1,\ldots,8$ (equivalently into the  cosets  of a normal subgroup ${\Gamma'}\cap{\Gamma(3)}$ in $\Gamma'$).  Then we decomposed each ${\mathcal{A}_k}\subset{\Gamma'}$ as ${\mathcal{A}_k}={\mathfrak{A}_k}\cup{\mathfrak{B}_k}$  according to $(5.6)$ and $(5.7)$. We have noticed that the set of pairs $(m_k,n_k)\in{\mathbb{Z}^2}$ with ${A^{m_k}B^{n_k}}\in{\mathfrak{A}_k}$ determines a sublattice $\widetilde{\mathcal{L}_k}$ of the lattice ${L_0}\cong{\mathbb{Z}^2}$.  Although the sublattice $\widetilde{\mathcal{L}_k}$ is not a dilate of  $L_0$  we may view it  as given by a  lattice equivalent to $L_0$ together with  the set of all of its appropriate half points. Such realization is not a  canonical one  and we have three ways, $I,II$ and $III$, to do this. In other words, we obtain the decomposition of $L_0$ into eight disjoint subsets, ${L_0}={\bigcup_{k=1}^{8}\widetilde{\mathcal{L}_k}}$,  each of which can be seen as
\begin{equation}
{\widetilde{\mathcal{L}_k}}={\{\widetilde{a}_k+\Lambda_l\}}\cup{\{\widetilde{a}_k+h(l)+\Lambda_l\}} \qquad   k=1,\ldots,8
\end{equation}
for $l=2,3,4$ (here ${\Lambda_l}={g^l\circ\Lambda_3}$ and $h(l)$ is an appropriate, depending on $l$, half-point of $\Lambda_l$). All lattices $\Lambda_l$'s are $4$-dilates  of $L_0$  and the essential differences between $I,II$ and $III$ lie in the different positions of half-points.
 These three realizations correspond to the elements of the group ${\left\langle g\right\rangle}<{SL_{2}\mathbb{Z}}$ involved in the formulae $(5.13)-(5.15)$. More precisely, although the three lattices ${\Lambda_l}=[\omega^l_1,\omega^l_2]={g^l\circ{\Lambda_3}}$, $l=2,3,4$   coincide (and  are all  4-dilates of the lattice $L_0$) the fact that $g\notin{\Gamma(2)}$ implies that the half-points of the lattices $g^{l}\circ{\Lambda_3}$'s are not preserved.  Since ${r_3(\left\langle g\right\rangle)}\cong{SL_2(3)/Q_8}$ we may view  the group ${Q_8}\triangleleft{SL_2(3)}$ as producing the decomposition ${L_0}=\bigcup_{k=1}^8{\widetilde{\mathcal{L}}_k}$ and we may view the quotient $SL_2(3)/{Q_8}$  (which is associated to the symmetries of the lattice $L_0$ described by the cyclic group $\left\langle g\right\rangle$) as responsible for the three realizations given by $(5.16), (5.17)$ and $(5.18)$ respectively. We see that:
\begin{itemize}
\item The realization $I$ is associated to $(\Lambda_3,\frac{\omega^3_1+\omega^3_2}{2})$ and involves the half points that correspond to the zeros of $\vartheta_3(v,\rho)$
\item The realization $II$ is associated to $(\Lambda_4,\frac{\omega^4_2}{2})$ and involves the half poins that correspond to the zeros of $\vartheta_4(v,\rho)$
\item The realization $III$ is associated to $(\Lambda_2,\frac{\omega^2_1}{2})$ and  involves the half points that correspond to the zeros of $\vartheta_2(v,\rho)$
\end{itemize}

Here $v=\frac{z}{4\omega_1}$, $L_0=[\omega_1,\omega_2]$,  $z=p(\tau)$ for $\tau\in\textit{H}$ and we use exactly the same subindex $l$ for a lattice $\Lambda_l$  and for the corresponding even theta function. Moreover, the relations $\Lambda_4=g\circ\Lambda_3$ and $\Lambda_2=g^{2}\circ\Lambda_3$ are parallel to the following relations respectively:
\begin{equation*}
{\vartheta_4^8(0,\tau)}={(j_g(\tau))^{-4}\vartheta^8_3(0,g\tau)}, \qquad {\vartheta^8_2(0,\tau)}={(j_{g^2}(\tau))^{-4}\vartheta^8_3(0,g^{2}\tau)}
\end{equation*}

Now, the global section $L(x')$ of lattices over $\textbf{T}^{*}$ introduced earlier leads to the fiber space over $\textbf{T}^{*}$ whose fiber at any point ${x'}\in{\textbf{T}^{*}}$ is a complex torus $\textbf{T}_{x'}$ given by $\mathbb{C}/L(x')$ and  attached to $x'$ at the origin. However although for each point $x'\in{\textbf{T}^{*}}$ the lattice $L(x')$ is well defined it is not equipped with any concrete basis and hence its half points are determined only up to permutations. The situation will change when we restrict ourselves to a single map $(U_\alpha,\tau_\alpha)$, that is  to  $x'\in{U_\alpha}\subset{\textbf{T}^{*}}$. Now we can write
\begin{equation*}
L(x')=L(\tau_{\alpha})=\mu(\tau_{\alpha})[1,\tau_{\alpha}]=[\omega^{\alpha}_1,\omega^{\alpha}_2]
\end{equation*}
and the half points are given by $h_1^{\alpha}=\frac{\omega^{\alpha}_1}{2}$, $h^{\alpha}_2=\frac{\omega^{\alpha}_2}{2}$ and by  $h^{\alpha}_3=\frac{\omega^{\alpha}_1+\omega^{\alpha}_2}{2}$ respectively. Since the decomposition of $L_0$ corresponds to the decomposition of $\mathbb{Z}^{2}$ given by $(5.11)$ the realization $I$ defines the decompositions of each  $L(\tau_{\alpha}(x'))$,  $x'\in{U_{\alpha}}$ onto eight sublattices together with their half points as follows:
\begin{equation}
{\widetilde{\mathcal{L}}_k^{\alpha}}={\{\widetilde{a^{\alpha}_k}+4L(\tau_{\alpha})\}}\cup{\{\widetilde{a}^{\alpha}_k+4h^{\alpha}_3+4L(\tau_{\alpha})\}},  \quad k=1,\ldots,8
\end{equation}

where ${\widetilde{a^{\alpha}_k}}={s_{k}\omega^{\alpha}_1}+t_{k}\omega^{\alpha}_2$. Each $\widetilde{\mathcal{L}}^{\alpha}_k$ produces torus isomorphic to $\textbf{T}(\tau_{\alpha}(x'))$ attached at the origin to $x'$ together with well defined half-point on it (corresponding to the zero of ${\vartheta_3(z,\tau_{\alpha})}:{U_{\alpha}}\times{\mathbb{C}}\rightarrow{\mathbb{C}}$). Since $k=1,\ldots,8$,  we may consider (on a set $U_{\alpha}$) the field $(4L(\tau_{\alpha}),4h^{\alpha}_3+4L(\tau_{\alpha}))^{\otimes8}$  and hence we  naturally obtain the function $\vartheta^8_3(0,\tau_{\alpha})$ (on $U_{\alpha}$).  Similarly, starting with the realization $II$ or $III$ we arrive to the functions $\vartheta^8_4(0,\tau_{\alpha})$ or to $\vartheta^8_2(0,\tau_{\alpha})$  respectively. None of these functions $\vartheta^8_l(0,\tau_{\alpha})$, $l=2,3,4$ can be  (using the atlas $\{(U_{\alpha},\tau_{\alpha})\}_{\alpha}$ ) extended to the whole $\textbf{T}^{*}$ to define any meaningful object on it.

 The existence of the three pictures $I,II$ and $III$ of each $\widetilde{L}_k$, $k=1,\ldots,8$ comes from the symmetry properties of the lattice ${p(\Gamma'{}\infty)}={L_0}$. Since the group $\left\langle g\right\rangle$ is responsible  for the existence of these three realizations  we may naturally involve the Hecke operators $T_{\left\langle g\right\rangle,k}$  introduced in the subsection $(3.2)$.  Thus, for  $l=2,3,4$ on each ${U_{\alpha}}\subset{\textbf{T}^{*}}$ we obtain
\begin{equation*}
 {T_{\left\langle g\right\rangle,4}\vartheta^8_l(0,\tau_{\alpha})}={\vartheta^8_3(0,\tau_{\alpha})+\vartheta^8_4(0,\tau_{\alpha})+\vartheta^8_2(0,\tau_{\alpha})}
\end{equation*} 

 Since for $x'\in{U_{\alpha}\cap{U_{\beta}}}$ we have  ${\tau_{\beta}(x')}={\gamma\tau_{\alpha}(x')}$  for some ${\gamma=\begin{pmatrix}a&b\\c&d\end{pmatrix}}\in{\Gamma'}$ and 
\begin{equation*}
{T_{\left\langle g\right\rangle,4}\vartheta^8_l(0,\tau_{\beta})}={(c\tau_{\alpha}+d)^{4}T_{\left\langle g\right\rangle,4}\vartheta^8_l(0,\tau_{\alpha})}
\end{equation*}
  the family ${\{T_{\left\langle g\right\rangle,4}\vartheta^8_l(0,\tau_{\alpha})\}}_{\alpha}$ forms well defined quadratic differential on $\textbf{T}^{*}$ which is exactly the same  for each $l=2,3,4$.  The necessity of applying on $U_{\alpha}$ the Hecke operator $T_{\left\langle g\right\rangle,4}$ to $\vartheta^8_l(0,\tau_{\alpha})$  (or equivalently, the necessity of taking equally weighted sum $\sum_{l=2}^{4}\vartheta^8_l(0,\tau_{\alpha})$ on $U_{\alpha}$) reflects the fact that each one of these three realizations is equally important.  Thus, the explicite forms of the Thetanullverte $\vartheta_l(0,\tau_{\alpha})$, $l=2,3,4$ on $U_{\alpha}$ (which result from these all three realizations)  provide
  \begin{equation}
 {T_{\left\langle g\right\rangle,4}\vartheta^8_l(0,\tau_{\alpha})}={\sum_{\underline{n}\in{\mathbb{Z}^8}}q_{\alpha}^{\underline{n}^2}+\sum_{\underline{n}\in{\mathbb{Z}^8}}(-1)^{\underline{n}\cdot{\underline{1}}}q_{\alpha}^{\underline{n}^2}+\sum_{\underline{n}\in{\mathbb{Z}^8}}q_{\alpha}^{(\underline{n}-\underline{e})^2}}
 \end{equation}
 and is further  equal to the following expresion
\begin{equation*}
{\sum_{\underline{n}\cdot\underline{1}\in{2\mathbb{Z}}}q_{\alpha}^{\underline{n}^2}+\sum_{\underline{n}\in{\mathbb{Z}^8}}q_{\alpha}^{(\underline{n}-\underline{e})^2}}=2\Theta_{E_8}(\tau_{\alpha})
\end{equation*}  

Here ${q_{\alpha}}={e^{i\pi\tau_{\alpha}}}$, ${\underline{e}}=(\frac{1}{2},\frac{1}{2},\ldots,\frac{1}{2})\in{\mathbb{Q}^8}$, and ${\underline{1}}=2\underline{e}$.

Another argument which leads to the sum of all $\vartheta^8_l(0,\tau_{\alpha})$, $l=2,3,4$ comes from the subsection $(3.2)$. Namely, for any holomorphic atlas on ${Y(2)}\cong{\textit{H}/\Gamma(2)}$ the transition functions preserve (pointwise) all half-points of each nonsingular fiber of the modular elliptic surface over $X(2)$. Hence ${\vartheta^8_3(0,\tau)}={\sum_{\underline{n}\in{\mathbb{Z}^8}}q^{\underline{n}^2}}$  is a $\Gamma(2)$-automorphic form of weight $4$.  Roughly speaking, the existence of a global section of half-points over the moduli space $Y(2)$  allows us to consider only the first part of the right side of $(6.3)$ wich contains only the lattice $\mathbb{Z}^8$. When we pass to the moduli space $\textbf{T}^{*}$ of complex tori, it is no longer possible and (on each $U_{\alpha}$) we must also involve the remaining terms of the left side of $(6.3)$, that is, we must consider the lattice $E_8$ instead of merely $\mathbb{Z}^8$ as for $Y(2)$.

Summerizing, the occurence of the $E_8$-symmetry related to the moduli space $\textbf{T}^{*}$ can be seen as a consequence of the relation between $\Gamma'$ and $Q_8$ coming from the modulo $3$ homomorphism and of the existence of the equally important three realizations of the decomposition of the lattice $L_0$ into 8 mutually disjoint subsets.

Moreover, from local relations  $\widetilde{H}_4(\vartheta^8_3(\tau)d\tau^2)\cong{const}\wp(z)dz^2$ coming off the subsection $(3.2)$  we obtain 
\begin{equation}
u(\tau_{\alpha})={\wp(p(\tau_{\alpha}),L_0)}={{const}\cdot\frac{\Theta_{E_8}(\tau_{\alpha})}{\eta^8(\tau_{\alpha})}} \qquad on\qquad U_{\alpha}
\end{equation}  

 and hence
 \begin{equation}
{\wp(p(\tau_{\alpha}),L_0)}={{const}{\cdot}q_{\alpha}^{-\frac{1}{3}}\sum_{n=0}^{\infty}\sum_{m=0}^{n}r_{E_8}(m)p_8(n-m)q_{\alpha}^n} 
 \end{equation}
 
 So, we may view the Weierstrass function on the  moduli curve  $\textbf{T}^{*}$ as a function which  encodes  the information about the decompositions of $L_0$.
 
 Let us notice the difference between the Jacobi  and our approach.  Although Jacobi forms involve both euclidean variable $z$ and hyperelliptic variable $\tau$ (in particular, the ratio of the Jacobi-Eisenstein forms of index $1$ and weight 10 and 12 respectively gives a constant multiple of the Weierstrass $\wp$-function for each $L_{\tau}$, $\tau\in{\textit{H}}$ ) in the Jacobi picture we must work with meromorphic functions on $\textit{H}\times{\mathbb{C}}$ satisfying some concrete conditions.  In our approach we simply translate the hyperbolic objects for $\Gamma'$, $\Gamma_c$, $\Gamma^{+}_{ns}(3)$, etc. into the euclidean objects on $\mathbb{C}-L_0$ and vice versa and this is a reason for the appearance of 8 sublattices of $L_0$ together with their appropriate half-points and further the appearance of $\Theta_{E_8}(q_{\alpha})$ on $U_{\alpha}$.
 
 We have shown that the  bridge between the hyperbolic structure of the universal covering space $\textit{H}$  of ${\textbf{T}^{*}}\cong{\textit{H}/\Gamma'}$ and the euclidean structure of $\mathbb{C}-L_0$, (${\textbf{T}^{*}}\cong{\mathbb{C}-L_0/L_0}$) is given by the function $\eta^4(\tau)$.  Now, rewriting the formula $(6.4)$ as 
 \begin{equation}
 {\wp(p(\tau_{\alpha}),L_0)\eta^8(\tau_{\alpha})}=const\cdot\Theta_{E_8}(\tau_{\alpha})
 \end{equation} 

we may view $\eta^8(\tau_{\alpha})$ as a bridge between $2$-periodic,with respect to $L_0$, function $\wp$  and the theta function of  the lattice $E_8$  (which may be produced by  the  decomposition of $L_0$ into 8 sublattices together with appropriate half points in three distict ways respectively).

Let us also notice that the function $\eta^8(\tau)$ provides very strong interrelation between the groups $\Gamma$ and $\Gamma'$. It is expressed by the fact that the ring of modular forms for $\Gamma$, that is the ring ${\mathfrak{M}(\Gamma)}={\mathbb{C}[g_2(\tau),g_3(\tau)]}$ can be written as ${\mathfrak{M}(\Gamma)}={\mathbb{C}[\eta^8(\tau)u(\tau),\eta^8(\tau)u'(\tau)]}$ and hence can be given as:
\begin{equation}
{\mathfrak{M}(\Gamma)}={\mathbb{C}[\eta^8(\tau)\wp(p(\tau),L_0),\eta^8(\tau)\wp'(p(\tau),L_0)]}
\end{equation}

This tells us that although the generators $g_2(\tau)$ and $g_3(\tau)$ are algebraically independent they produce differentials $\wp(z)dz^2$ and $\wp'(z)dz^3$ respectively and hence we have some ``elliptic'' type of differential relation between them.

\end{document}